\documentclass[reqno,10pt]{amsart}
\usepackage{color}

\usepackage{mathtools}
\usepackage{amsfonts}
\usepackage{amssymb}
\usepackage{amsthm}
\usepackage{newlfont}
\usepackage{graphicx}
\usepackage{float}
\usepackage{amsmath}
\usepackage{extarrows}
\usepackage{titleps}
\usepackage{booktabs}
\usepackage{geometry}
\usepackage{tikz}
\usetikzlibrary{matrix,shapes,arrows,positioning,chains}

\numberwithin{equation}{section}
\newtheorem{thm}{Theorem}[section]

\newtheorem{rem}{Remark}[section]
\newtheorem{prop}{Proposition}[section]

\newtheorem*{rem*}{Remark}

\newcommand{\ed}{\end {document}}

\begin{document}
\title[Helmholtz]{\bf Localization for general Helmholtz}

\author[X.Y. Cheng]{Xinyu Cheng}
\address{Xinyu Cheng, Research Institute of Intelligent Complex Systems, Fudan University, Shanghai, P. R. China} 
\email{xycheng@fudan.edu.cn}

\author[Dong Li]{Dong Li}
\address{Dong Li,  Department of Mathematics, the University of Hong Kong, Hong Kong, China}
\email{mathdl@hku.hk}

\author[W. Yang]{ Wen Yang}
\address{Wen Yang, Department of Mathematics, Faculty of Science and Technology, University of Macau, Taipa, Macau, China}
\email{wenyang@um.edu.mo}
	
\thanks{X.Y. Cheng is partially supported by the Shanghai “Super Postdoc” Incentive Plan (No. 2021014), the International Postdoctoral Exchange Fellowship Program (No. YJ20220071) and China Postdoctoral Science Foundation (Grant No. 2022M710796,2022T150139). 
Dong Li is supported  in part by NSFC 12271236. 
W. Yang is partially supported by National Key R\&D Program of China 2022YFA1006800, NSFC No. 12171456, NSFC No. 12271369 and SRG2023-00067-FST.}

\renewcommand{\thefootnote}{\fnsymbol{footnote}}
	
\maketitle

{\noindent\small{\bf Abstract:}
In \cite{gmw2022}, Guan, Murugan and Wei established the equivalence of the classical
Helmholtz equation with a ``fractional Helmholtz" equation in which  the Laplacian
operator is replaced by the nonlocal fractional Laplacian operator. More general equivalence
results are obtained for symbols which are complete Bernstein and satisfy additional regularity
conditions. In this work we introduce a novel and general set-up for this Helmholtz equivalence problem.
We show that under very mild and easy-to-check conditions on the Fourier multiplier, the general Helmholtz equation
can be effectively reduced to a  localization statement on the support of the symbol.
}
	
\vspace{1ex}

\section{Introduction}
A classical problem in mathematical physics is to find eigen-pairs to the linear problem $\mathcal A f = \lambda f$ where
$\mathcal A$ is a linear operator (bounded or unbounded) acting on some general function spaces. If $\mathcal A =-\Delta$, where $\Delta = \sum_i \frac {\partial^2} {\partial {x_i }^2}$ with various boundary
conditions, this becomes the standard Helmholtz equation. The Helmholtz equation
 appears ubiquitously in physical
and engineering applications such as acoustic radiation,  heat conduction, propagation of water waves etc.
In the context of heat or wave equations,  the classical Helmholtz system naturally arises as the time-independent eigen-pair form when one employs the traditional separation of variable method. Recently in \cite{wva2020} a scalar fractional
Helmholtz type equation is derived from the Maxwell's equations by incorporating nonlocal long range effects.
In this connection a prototypical version of the fractional Helmholtz equation amounts to taking $\mathcal A =(-\Delta)^s$ for some $s>0$.   A first subtle mathematical issue is to understand the limit transitions $s \to 1$ (harmonic)
or $s\to n \in \mathbb N$ (poly-harmonic) in various normed functional spaces.  Another circle of questions, recently
initiated in the work of Guan, Murugan and Wei \cite{gmw2022}, is concerned with the deep connection between
the classical Helmholtz equation and the fractional ones, namely (taking $\lambda=1$):
\begin{align}
\{ f:\, -\Delta f = f \}  \qquad \text{vs} \qquad \{ g: \; (-\Delta)^s g = g \};
\end{align}
 or more generally
 \begin{align}
\{ f:\, -\Delta f = f \}  \qquad \text{vs} \qquad \{ g: \; \Phi(-\Delta) g =  \Phi(1) g \};
\end{align}
where $\Phi(-\Delta)$ is a suitable Dirichlet-to-Neumann operator corresponding to a suitable harmonic extension
problem involving the symbol $\Phi(\xi^2)$.

 In \cite{gmw2022}, by deeply exploiting the harmonic extension technique,  Guan, Murugan and Wei established
 the
 following set of rigidity type results:
\begin{enumerate}
\item If $0<s<2$, $d=1$, $u\in L^{\infty}(\mathbb R)$ satisfies $\Lambda^s u = u$,  then\footnote{The case $d=1$ was first obtained by Fall and Weth in \cite{fw2016}.}
$u(x) =c_1
\cos x + c_2 \sin x$.
\item If $0<s~{\leq}~2$, $d\ge 2$, $u \in C^{\infty}(\mathbb R^d) \cap L^{\infty}(\mathbb R^d)$ satisfies
$\Lambda^s u = u$ and $\lim\limits_{|x| \to \infty} u (x) =0$, then $-\Delta u = u$.
\item If $m\in\mathbb{N}$, $d\ge 2$, $u \in C^{\infty}(\mathbb R^d) \cap L^{\infty}(\mathbb R^d)$ satisfies
$(-\Delta)^m u = u$ if and only if $-\Delta u = u$.
\item  Consider $d\ge 2$ and $\Phi(-\Delta) u = \Phi(1) u$ in $\mathbb R^d$ where
$\Phi$ is a complete Bernstein function. Assume $u(x) \to 0$
as $|x|\to \infty$. Consider the associated extension problem: $\tilde u \in H^1_{\mathrm{loc}}(
\mathbb R^{d+1}_+, a(t) )$ solves
\begin{align*}
\begin{cases}
\nabla_{t,x} \cdot ( a(t) \nabla_{t,x} \tilde u) =0, \qquad \text{on $\mathbb R^{d+1}_+$};\\
\lim\limits_{t\to 0} a(t) \partial_t u = -c_{n,a} \Phi(-\Delta) u, \quad \text{on $\mathbb R^d$}; \\
\lim\limits_{t\to 0} \tilde u =u;
\end{cases}
\end{align*}
where one assumes the weight $a(t)$ for $\Phi(-\Delta)$ is $A_2$ and obeys $a(t) \sim t^{\alpha}$ for
$t\gg 1$ and $|\alpha|<1$.  Under the above conditions, one has
\begin{align} \label{1.4a}
\boxed{\text{
  $u \in L^{\infty}(\mathbb R^d)$ with $u(\infty) =0$ solves $
\Phi(-\Delta ) u = \Phi(1) u$ in $\mathbb R^d$ if and only if $-\Delta u = u $ in $\mathbb R^d$}. }
\end{align}
\end{enumerate}

As was somewhat hinted earlier, the proof in \cite{gmw2022} hinges on the machinery of harmonic extension which poses
various subtle technical restrictions. For example, in dimensions $d\ge 2$, one needs to impose the decay
condition $u(x)\to 0$ as $|x|\to \infty$ for the fractional Helmholtz problem $(-\Delta)^s u = u$, $0<s<2$.
Besides, some additional asymptotic conditions need to be imposed on the weight function $a(t)$ for the Bernstein Helmholtz problem. {In recent \cite{cly2022}, we removed the decay condition and established the equivalency of classical Helmholtz equation and the
fractional Helmholtz equation corresponding to the operator $(-\Delta)^s$, $0<s<2$. However, the proof in \cite{cly2022}
hinges on the special form of the operator $(-\Delta)^s$ which has no obvious bearing on the general case.
}

In this work  we shall remove all these aforementioned technical restrictions by
 developing a novel and general set up for this Helmholtz equivalence problem.
We shall consider the general problem
\begin{align} \label{1.5a}
\Phi(-\Delta) u = \Phi(1) u \quad \text{in} \quad S^{\prime}(\mathbb R^d),
\end{align}
where $u \in L^{\infty}(\mathbb R^d)$.

We make the following technical assumptions on the function $\Phi:\, [0, \infty) \to \mathbb R$:
\begin{itemize}
\item[(a)] \underline{Smoothness and mild growth at $z=\infty$}. We assume $\Phi \in
C([0,\infty) ) \cap C^{\infty}((0,\infty))$, and all derivatives of $\Phi$ are polynomially bounded as $z\to \infty$: namely for some $z_0\ge 2$,
it holds that for all $k\ge 0$,
\begin{align}
\label{2.25aa}
|\partial^k \Phi (z) | \le C_{\Phi, k, z_0} z^{n_k}, \qquad \forall\, z\ge z_0,
\end{align}
where $C_{\Phi, k,z_0}>0$ is a constant depending only on ($\Phi$, $k$, $z_0$), and $n_k$ depends on $k$.

\item [(b)] \underline{Mild singularity at $z=0$}. 
For
some $0<\varepsilon_0\le \frac 12$,
\begin{align} \label{2.25a}
\sum_{j=0}^{d+1} \int_0^{\varepsilon_0} |z|^j | (\partial_z^j \Phi)(z)| \cdot \frac{ dz}z  <\infty.
\end{align}
Here note that in \eqref{2.25a}, the condition for $j=0$ reads as
$\int_0^{\varepsilon_0} |\Phi(z)| z^{-1} dz <\infty$. This implies
$\Phi(0) =0$ by continuity of $\Phi$.
\item [(c)] \underline{Smooth uni-valence at $z=1$}.  We assume $\Phi^{\prime}(1) \ne 0$ and $
\Phi(t) \ne \Phi(1)$ for any $t \in (0,1)\cup (1,\infty)$.
\end{itemize}

\begin{rem}
A prototypical $\Phi$ satisfying conditions (a)-(c)  is $\Phi(z) = z^{s}$ with $s>0$.
This corresponds to the usual fractional Laplacian case.  Other nontrivial examples are:
\begin{align*}
		& \Phi(z)= (m^2+ z)^{\frac s 2} -m, \qquad s>0; \\
		& \Phi (z) = z^{\frac 12} \tanh( z^{\frac 12});\\
		& \Phi(z)= z^{\frac 12} ( \tanh(z^{\frac 12}) )^{-1}.
\end{align*}
The Fourier symbols $|\xi| \tanh(|\xi|)$, $|\xi| (\tanh(|\xi|) )^{-1}$ play important roles in the Dirichlet to Neumann
	map theory in connection with linear water waves.

One should note that for $x>0$:
\begin{align*}
{\frac d {dx} \left( \frac {\tanh x} x \right)= -
\frac {\tanh x} {x^2} + \frac 1{ x \cosh^2 x} =
\frac 1{x^2 {\cosh ^2x}} \Bigl( x - \frac 12 \sinh ( 2x)  \Bigr) }<0.
\end{align*}
Thus $h(x) = \frac {x} {\tanh x} $ is strictly increasing.  
\end{rem}
\begin{rem}
Interestingly a class of nonlocal Helmholtz problems with nonlocal ``completely Bernstein" symbols 
 are also treated in recent \cite{David23} (see also \cite{BS22}). 
\end{rem}

Concerning \eqref{1.5a}, an immediate subtle technical issue is to ensure $\Phi(-\Delta) u \in \mathcal S^{\prime}(\mathbb R^d)$ for $u \in L^{\infty}(\mathbb R^d)$.  The next proposition clarifies this point. We first
examine the action of $\Phi(-\Delta)$ on test functions which are Schwartz.

\begin{prop} \label{pr1.1}
Suppose $\Phi \in C([0,\infty) ) \cap C^{\infty}((0, \infty) )$ satisfies the conditions (a)--(c). For $\psi
\in \mathcal S(\mathbb R^d)$, define
\begin{align*}
\widehat{ \Phi(-\Delta) \psi} (\xi) = \Phi (|\xi|^2) \widehat{\psi}(\xi), \qquad \xi \in \mathbb R^d.
\end{align*}
Then $\Phi(-\Delta) \psi \in L^1(\mathbb R^d)$ and  for some integers $k_1\ge 0$, $k_2\ge 0$,
\begin{align*}
\| \Phi(-\Delta) \psi \|_{L_x^1(\mathbb R^d)} \le C_{d,k_1, k_2, \Phi}
\sum_{|\alpha|\le k_2} \| (1+|x|^2)^{k_1} \partial^{\alpha} \psi \|_{L_x^{\infty}(\mathbb R^d)},
\end{align*}
where $C_{d, k_1, k_2, \Phi}>0$ is a constant depending only
on ($d$, $k_1$, $k_2$, $\Phi$).
\end{prop}
The proof of Proposition \ref{pr1.1} is given in Section 2.

Thanks to Proposition \ref{pr1.1}, for $u \in L_x^{\infty}(\mathbb R^d)$ we can define $\Phi(-\Delta) u \in \mathcal S^{\prime}(\mathbb R^d)$ as an element in $S^{\prime}(\mathbb R^d)$ via the following
\begin{align} \label{1.7aa}
\Bigl( \Phi(-\Delta) u \Bigr) ( \varphi) = \int_{\mathbb R^d}
u \Phi(-\Delta) \varphi dx, \qquad\forall\, \varphi \in \mathcal S(\mathbb R^d).
\end{align}

Alternatively and equivalently, we extend the usual $L^2$-pairing of Schwartz functions (see \eqref{L2pairing}) to 
 $\mathcal S^{\prime}(\mathbb R^d)$-$\mathcal S(\mathbb R^d)$ pairing as the following
(below $\overline{z}$ denotes the usual complex conjugate for $z\in \mathbb C$)
:
\begin{align} \label{1.8bb}
\langle \Phi(-\Delta)u, \psi \rangle := \int_{\mathbb R^d} u(x)  \overline{(\Phi(-\Delta) \psi)(x)} dx,
\qquad \forall\, \psi \in \mathcal S(\mathbb R^d).
\end{align}
In particular we have the estimate
\begin{align*}
|\langle \Phi(-\Delta)u, \psi \rangle| \lesssim  \| u \|_{L_x^{\infty}(\mathbb R^d)} \sum_{|\alpha|\le k_2} \| (1+|x|^2)^{k_1}
\partial^{\alpha} \psi\|_{L_x^{\infty}(\mathbb R^d)}.
\end{align*}
Thus in its natural weak formulation, \eqref{1.5a} reads as:
\begin{align}
\label{1.16a}
\text{$u \in L_x^{\infty}(\mathbb R^d)$~\ satisfies}~\ \langle u, \Phi(-\Delta) \psi \rangle = \langle u,  \Phi(1) \psi \rangle, \quad \forall\, \psi \in
\mathcal S(\mathbb R^d)~\mathrm{and}~\Phi~\mbox{satisfies (a)--(c)}.
\end{align}

\begin{rem}
The main difference between \eqref{1.7aa} and \eqref{1.8bb} is that the former is linear in the test function $\varphi$
whereas the latter is \emph{conjugate linear} in the test function $\psi$.
\end{rem}
\begin{rem}
  To put things into perspective, one should note
that in general it is a subtle issue to define   the action of fractional Laplacian on a 
general tempered distribution u,   since the fractional 
Laplacian multiplier does not preserve the Schwartz space. 
What we basically show is that for fractional Laplacian operators and slightly more general symbols,  one can start by carefully defining the action of the fractional operator on the Schwartz function, show that the corresponding norm depends only on the weighted Sobolev 
norm of the test function; then in a natural way one can define for (say) bounded function u, the corresponding tempered distribution  $(-\Delta)^s u$.    
The only place where Fourier transform enter,  is in the action of the symbol
on the Schwartz test functions.

\end{rem}

The main result of this paper is the following.

\begin{thm}[Classification of general Helmholtz, case $\Phi^{\prime}(1)\ne 0$] \label{thm1}
Let $d\ge 1$.
Suppose  $u \in L^{\infty}(\mathbb R^d)$ solves the problem \eqref{1.16a}.
Then the following hold:

\begin{enumerate}
\item $\mathrm{supp}(\widehat{u} ) \subset
K=\{ \xi:\; \text{$\xi=0$ or $|\xi|=1$} \}.$ More precisely we have
\begin{align} \label{3.2a}
\langle \widehat{u}, \phi \rangle =0, \qquad  \forall\, \phi \in C_c^{\infty}( \mathbb R^d \setminus  K).
\end{align}

\item If $\Phi(1) \ne 0$, then $ \mathrm{supp}(\widehat{u} )
\subset \{ \xi: \, |\xi| =1 \}$, and
\begin{align} \label{3.3a}
\langle \widehat{u}, \phi \rangle =0, \qquad
\forall\, \phi \in C_c^{\infty}(\mathbb R^d \setminus \{ \xi:\, |\xi|=1 \} ).
\end{align}
Furthermore we have in this case,
\begin{align}
\left\langle u, \mathcal F^{-1}\Bigl(  ( |\xi|^2-1) \psi(\xi)   \Bigr) \right\rangle =0, \qquad \forall\, \psi \in \mathcal S(\mathbb R^d).
\end{align}
\end{enumerate}
\end{thm}
\begin{rem}
Since $\widehat u$ is compactly supported, the function $u$ can be identified as a $C^{\infty}(\mathbb R^d)$ function
thanks to Paley-Wiener.
Moreover in the case $\Phi(1) \ne 0$, we conclude that $u$ solves the classical Helmholtz equation, namely
\begin{align*}
-\Delta u = u \quad \text{in}\quad \mathbb R^d.
\end{align*}
\end{rem}

In Section 4 of this paper, we shall show that our theorem above already covers the general
Bernstein Helmholtz case in \cite{gmw2022}.  As was already mentioned, the conditions on
the symbol $\Phi(z)$ is already quite general and easy-to-check in practice.

\medskip

\medskip

On the other hand, as the avid reader may notice, the condition $\Phi^{\prime}(1) \ne 0$ may appear a bit restrictive. As a matter of
fact, one can also consider a slightly more non-degenerate condition.  That is, we assume $\Phi(1)\ne 0$, and for some $j_0\ge 1$, it holds that
\begin{align*}
\frac {d^j}{dz^j} \Phi(z) \Bigr|_{z=1}=: \Phi^{(j)}(1) =0, ~ 1\le j\le j_0-1, \quad\mathrm{and}\quad  \Phi^{(j_0)}(1) \ne 0.
\end{align*}
If $j_0=1$, we simply require $\Phi^{\prime}(1)\ne 0$.   For $j_0\ge 2$, $\Phi^{(j_0)}(1)$ is the first
nonzero coefficient after $\Phi^{\prime}(1)=0$.  As we shall show momentarily, under this general condition  the corresponding localization statement takes a slightly different form, namely: the Helmholtz equation
\begin{align*}
\Phi(-\Delta) u = \Phi(1) u
\end{align*}
can be effectively reduced to
\begin{align*}
(-\Delta -1)^{j_0}  u=0.
\end{align*}

For the sake of completeness, we now record the aforementioned slightly more general  technical assumptions on the function $\Phi:\, [0, \infty) \to \mathbb R$ as follows.
\begin{itemize}
\item[(c1)] \underline{Smoothness and mild growth at $z=\infty$}. We assume $\Phi \in
C([0,\infty) ) \cap C^{\infty}((0,\infty))$, and all derivatives of $\Phi$ are polynomially bounded as $z\to \infty$: namely for some $z_0\ge 2$,
it holds that for all $k\ge 0$,
\begin{align}  \label{1.26a}
|\partial^k \Phi (z) | \le C_{\Phi, k, z_0} z^{n_k}, \qquad \forall\, z\ge z_0,
\end{align}
where $C_{\Phi, k,z_0}>0$ is a constant depending only on ($\Phi$, $k$, $z_0$), and $n_k$ depends on $k$.

\item [(c2)] \underline{Mild singularity at $z=0$}. 
For
some $0<\varepsilon_0\le \frac 12$,
\begin{align}
\sum_{j=0}^{d+1} \int_0^{\varepsilon_0} |z|^j | (\partial_z^j \Phi)(z)| \cdot \frac{ dz}z  <\infty.
\end{align}
\item [(c3)] \underline{Non-degeneracy at $z=1$}.  We assume $\Phi(1) \ne 0$ and $
\Phi(t) \ne \Phi(1)$ for any $t \in (0,1)\cup (1,\infty)$.  Furthermore we assume for some integer $j_0\ge 1$,
it holds that
\begin{align} \label{1.26c}
\Phi^{(j)}(1) =0, ~ 1\le j\le j_0-1,\quad\mbox{and}\quad \Phi^{(j_0)}(1) \ne 0.
\end{align}
If $j_0=1$, we simply require $\Phi^{\prime}(1)\ne 0$.
\end{itemize}

\begin{thm}[Classification of general Helmholtz, case $\Phi^{(j_0)}(1)\ne 0$] \label{thm2}
Let $d\ge 1$.
Suppose  $u \in L^{\infty}(\mathbb R^d)$ solves the equation
\begin{align}
\langle u, \Phi(-\Delta) \psi \rangle= \langle u, \Phi(1) \psi\rangle, \qquad\forall\, \psi
\in \mathcal S(\mathbb R^d);
\end{align}
where $\Phi:\, [0,\infty)\to \mathbb R$ satisfies the conditions (c1)--(c3) (as specified in
\eqref{1.26a}--\eqref{1.26c}).
Then the following hold:

We have $ \mathrm{supp}(\widehat{u} )
\subset \{ \xi: \, |\xi| =1 \}$, and
\begin{align} \label{3.3a}
\langle \widehat{u}, \phi \rangle =0, \qquad
\forall\, \phi \in C_c^{\infty}(\mathbb R^d \setminus \{ \xi:\, |\xi|=1 \} ).
\end{align}
Furthermore we have in this case (below $j_0$ is the same integer as in \eqref{1.26c}),
\begin{align}
\left\langle u, \mathcal F^{-1}\Bigl(  ( |\xi|^2-1)^{j_0} \psi(\xi)   \Bigr) \right\rangle =0, \qquad \forall\, \psi \in \mathcal S(\mathbb R^d).
\end{align}
In yet other words, $u \in L^{\infty}(\mathbb R^d)$ can be identified as a $C^{\infty}$ function and
\begin{align}
(-\Delta -1)^{j_0} u=0.
\end{align}
\end{thm}

\medskip

In the following subsection we fix some notation used throughout this paper.
\subsection*{Notation} For any two nonnegative quantities $X$ and $Y$, we write $X\lesssim Y$ or
$Y\gtrsim X$ if $X\le CY$ for some
harmless constant $C>0$. We write $X\ll Y$ or $Y\gg X$ if $X\le c Y$ for some sufficiently small constant $c>0$.
The needed smallness of the constant $c$ is usually clear from the context.

We denote by $\mathcal S(\mathbb R^d)=\mathcal S(\mathbb R^d\to \mathbb C)$ the usual space of complex-valued Schwartz functions and
$\mathcal S^{\prime}(\mathbb R^d)$ the space of tempered distributions.

For $u \in \mathcal S(\mathbb R^d)$,  we adopt the following convention for Fourier transform:
\begin{align*}
(\mathcal F u)(\xi) =\widehat u (\xi) = \int_{\mathbb R^d} u(y) e^{-iy\cdot \xi} dy \quad\mathrm{and}\quad
u(x) = \frac 1 {(2\pi)^d} \int_{\mathbb R^d} \widehat{u}(\xi) e^{i \xi \cdot x} d\xi
=: (\mathcal F^{-1} \widehat u )(x).
\end{align*}
Let $s>0$.  For $u\in \mathcal
S(\mathbb R^d)$, $d\ge 1$, the fractional Laplacian $\Lambda^s u
=(-\Delta)^{\frac s2} u $
is defined via Fourier transform as
\begin{align*}
\widehat{\Lambda^s u } (\xi) = |\xi|^s \widehat{u}(\xi), \qquad \xi \in \mathbb R^d.
\end{align*}
For $f_1:\, \mathbb R^d \to \mathbb C$, $f_2:\, \mathbb R^d \to \mathbb C$,  $f_1$, $f_2$ Schwartz, we denote
the usual $L^2$ pairing:
\begin{align}\label{L2pairing}
\langle f_1, f_2\rangle : = \int_{\mathbb R^d} f_1(x) \overline{f_2(x) } dx,
\end{align}
where $\overline{z}$ denotes the usual complex conjugate of $z\in \mathbb C$.  The usual
Plancherel formula reads
\begin{align*}
\langle \widehat f_1, \widehat f_2 \rangle = (2\pi)^d \langle f_1, f_2 \rangle.
\end{align*}
If we denote $f_3 =\widehat f_2$, then $f_2 = \mathcal F^{-1} (f_3)$. Thus for
$f_1$, $f_3 \in \mathcal S(\mathbb R^d)$, it holds that
\begin{align*}
\langle \widehat f_1, f_3 \rangle = (2\pi)^d \langle f_1, \mathcal F^{-1} (f_3) \rangle.
\end{align*}
More generally for tempered distribution $u$, we have
\begin{align*}
\langle \widehat u, \phi \rangle = (2\pi)^d \langle u, \mathcal F^{-1}(\phi) \rangle, \qquad
\forall\, \phi \in \mathcal S(\mathbb R^d).
\end{align*}

\section{Proof of Proposition \ref{pr1.1}}
\begin{proof}[Proof of Proposition \ref{pr1.1}]
We begin by noting that in the regime $t \gtrsim 1$, the function $\Phi(t)$ along with its derivatives
grow at most polynomially.  Let $\chi_{|\xi| \gtrsim 1}$ be a smooth cut-off function
localized to the regime $|\xi| \gtrsim 1$.
Since $\widehat \psi \in \mathcal S(\mathbb R^d)$,  it is not difficult to
check that $\mathcal F^{-1} ( \Phi(|\xi|^2) \chi_{|\xi|\gtrsim 1} \widehat{\psi} )
\in L_x^1(\mathbb R^d)$ and
\begin{align}
\| \mathcal F^{-1} ( \Phi(|\xi|^2) \chi_{|\xi|\gtrsim 1} \widehat{\psi} ) \|_{L_x^1(\mathbb R^d)}
\lesssim
\sum_{|\alpha|\le k_2} \| (1+|x|^2)^{k_1} \partial^{\alpha} \psi \|_{L^{\infty}(\mathbb R^d)},
\end{align}
where $k_1\ge 0$, $k_2\ge 0$ are integers.

It suffices for us to examine the piece
\begin{align}
\beta(x) = \int_{\mathbb R^d} \chi(\xi) \Phi(|\xi|^2) \widehat {\psi }(\xi) e^{i \xi \cdot x} d \xi,
\end{align}
where $\chi \in C_c^{\infty}(\mathbb R^d)$ is a radial bump function localized to $\{\xi:\, |\xi| \ll 1 \}$.
Clearly
\begin{align}
\| \beta \|_{L_x^1(\mathbb R^d)} \lesssim \| \beta_1 \|_{L_x^1(\mathbb R^d)} \| \psi\|_{L_x^1(\mathbb R^d)},
\end{align}
where
\begin{align}
\beta_1(x) = \int_{\mathbb R^d} \chi(\xi) \Phi(|\xi|^2)  e^{i \xi \cdot x} d \xi.
\end{align}
Thus to finish the proof of Proposition \ref{pr1.1}, we only need to prove the next proposition.
\end{proof}
\begin{prop} \label{pr2.1}
Let $\varepsilon_0$ be a small positive number such that $\chi$ is supported in $\{x\in\mathbb{R}^d\mid |x|\leq \varepsilon_0\}$. Then we have
\begin{align} \label{2.5a}
\boxed{
\left\| \mathcal F^{-1} \Bigl(
\Phi (|\xi|^2) \chi (\xi)  \Bigr) \right\|_{L_x^1(\mathbb R^d)}
\lesssim \sum_{j=0}^{d+1} \int_0^{\varepsilon_0} |z|^j | (\partial_z^j \Phi)(z)|
\cdot \frac{ dz}z.}
\end{align}
\end{prop}
\begin{proof}[Proof of Proposition \ref{pr2.1}]
The case for dimension $d=1$ is left to the reader as an exercise.

We now consider for example for $d=3$, denoting $r=|x|\ge r_0\gg 1$ and $\rho =|\xi|$, we have
(below we slightly abuse the notation and still denote $\chi(\rho)=\chi(\xi)$)
\begin{equation}
\begin{aligned}
|\beta_1(x)| &\lesssim  \Bigl| \int_0^{\varepsilon_0} \chi (\rho) \Phi(\rho^2) \frac {\sin \rho r} {\rho r} \rho^2 d\rho \Bigr| \\
& \lesssim  \Bigl|
\underbrace{\int_0^{\varepsilon_0} \chi (\rho) \chi_1(r \rho) \Phi(\rho^2) \frac {\sin \rho r} {\rho r} \rho^2 d\rho}_{=:I_1(r)} \Bigr| + \Bigl| \underbrace{ \int_0^{\varepsilon_0} \chi (\rho) (1-\chi_1(r\rho) ) \Phi(\rho^2) \frac {\sin \rho r} {\rho r} \rho^2 d\rho
}_{=:I_2(r)} \Bigr|.
\end{aligned}
\end{equation}
In the above $\chi_1 \in C_c^{\infty}(\mathbb R)$ is an even function such that
$\chi_1(z)=1$ for $|z|\le 0.9$ and $ \chi_1(z)=0$ for $|z|\ge 1$.
Clearly
\begin{align}
\int_{r_0}^{\infty}| I_1(r) | r^2 dr \lesssim
\int_0^{\varepsilon_0} \chi (\rho ) |\Phi(\rho^2)| \rho^2 \Bigl( \int_{r\lesssim \rho^{-1}} r^2 dr \Bigr) d\rho \lesssim \int_0^{\varepsilon_0} |\Phi(\rho^2) | \frac {d\rho} {\rho}
\lesssim \int_0^{\varepsilon_0^2} | \Phi(z) | z^{-1} dz.
\end{align}
On the other hand,
\begin{align}
I_2(r) & =\frac 1r \int_0^{\varepsilon_0} \rho \chi(\rho) (1-\chi_1(r\rho) ) \Phi(\rho^2) \sin (\rho r) d\rho.
\end{align}
By using successive integration by parts, we have
\begin{align} \label{2.10a}
|I_2(r)| \lesssim
\frac 1 {r^4} \int_0^{\varepsilon_0} \Bigl|
\frac {d^3}{d\rho^3}
\left( \rho \chi(\rho) (1-\chi_1(r\rho) ) \Phi(\rho^2) \right) \Bigr| d\rho.
\end{align}
Thus
\begin{align}
\int_{r_0}^{\infty} |I_2(r)| r^2 dr \lesssim~ \text{R.H.S. of} ~\eqref{2.5a}.
\end{align}
Adding the estimates for $I_1(r)$ and $I_2(r)$, we obtain the proof for the case $d=3$.

For the general case $d\ge 2$, we only need to work with the expression
\begin{align}
\int_{r_0}^{\infty}
\Bigl| \int_0^{\varepsilon_0} \chi (\rho) \Phi(\rho^2) F_d (\rho r) \rho^{d-1} d\rho \Bigr| r^{d-1} dr,
\end{align}
where
\begin{align*}
F_d(\lambda) = \int_{\mathbb S^{d-1} } e^{i \lambda e_1 \cdot \omega} d\sigma(\omega).
\end{align*}
In Subsection 2.1, we collect some standard material on the Bessel functions and some needed auxiliary
estimates on the function $F_d(\lambda)$.

The regime $\rho r \lesssim 1$ is clearly under control, i.e.
\begin{align}
\int_{r_0}^{\infty}
\Bigl| \int_0^{\varepsilon_0} \chi (\rho) \Phi(\rho^2) F_d (\rho r)  \chi_{r \rho \lesssim 1} \rho^{d-1} d\rho \Bigr| r^{d-1} dr
\lesssim\;   \int_0^{\varepsilon_0} \rho^{-1}\chi(\rho) |\Phi(\rho^2)| d\rho \lesssim \int_0^{\varepsilon_0^2} | \Phi(z) | z^{-1} dz.
\end{align}
On the other hand, for $\lambda=\rho r \gg 1$,  we note that by \eqref{2.20a} and taking $K$ large
\begin{align*}
\left| F_{d}(\lambda) - \text{finitely many terms of the form $\lambda^{-\alpha} e^{i\lambda}$}
\right| \lesssim \lambda^{-K}.
\end{align*}
Clearly the error term is under control:
\begin{align}
\int_{r_0}^{\infty}
 \int_0^{\varepsilon_0} \chi (\rho) |\Phi(\rho^2) | (r \rho)^{-K}  \chi_{r \rho \gg 1} \rho^{d-1} d\rho  r^{d-1} dr
\lesssim\;   \int_0^{\varepsilon_0} \chi(\rho) |\Phi(\rho^2)| \frac {d\rho} {\rho} \lesssim \int_0^{\varepsilon_0^2} | \Phi(z) | z^{-1} dz.
\end{align}
It remains to treat the terms
\begin{align}
\int_{r_0}^{\infty}
\Bigl| \int_0^{\varepsilon_0} \chi (\rho) \Phi(\rho^2)  (r \rho)^{-\alpha} e^{ir \rho}  \chi_{r \rho \gg 1} \rho^{d-1} d\rho \Bigr| r^{d-1} dr.
\end{align}
One can perform successive integration by parts in much the same way as in \eqref{2.10a}.
We omit the details.
\end{proof}

\subsection{Bessel functions and auxiliary estimates for the function $F_d(\lambda)$}
For $\nu>-\frac 12$, we recall the following formula for the standard Bessel function $J_{\nu}$
\begin{align}
J_{\nu}(\lambda)
= \frac 1 {2^{\nu} \Gamma(\nu+\frac 12) \sqrt{\pi}}
\lambda^{\nu} \int_{-1}^1 e^{i \lambda t} (1-t^2)^{\nu-\frac 12} dt, \qquad \lambda> 0.
\end{align}
We shall need the well-known asymptotic formula for $J_{\nu}(\lambda)$ (see \cite[Section 17.5]{ww1920} or \cite{Bformula}):
here we assume $\nu\ge 0$, $\lambda \gg 1$,  then
\begin{align}
J_{\nu}(\lambda) \sim \left(\frac 2 {\pi \lambda}  \right)^{\frac 12}
\Bigl( \cos \omega_{\lambda}
\sum_{k=0}^{\infty} (-1)^k \frac {a_{2k}(\nu)} {\lambda^{2k}}
-\sin \omega_{\lambda} \sum_{k=0}^{\infty}
(-1)^k \frac {a_{2k+1}(\nu)} {\lambda^{2k+1}} \Bigr),
\end{align}
where
\begin{align*}
\omega_{\lambda}= \lambda -\frac 12 \nu \pi- \frac 14 \pi, \quad \mbox{and}\quad
a_k(\nu)
= \frac {(4\nu^2-1) (4\nu^2-3^2) \cdots (4\nu^2-(2k-1)^2)} {k! 8^k}.
\end{align*}
In particular, for any integer $K\ge 1$, we have
\begin{align} \label{2.15a}
\boxed{
\left| J_{\nu}(\lambda)
-\lambda^{-\frac 12} \cos \lambda \sum_{j=0}^K
\alpha_{j,\nu} \lambda^{-j}
-\lambda^{-\frac 12}\sin \lambda \sum_{j=0}^K
\beta_{j,\nu} \lambda^{-j} \right| \le C_{K, \nu} \lambda^{-K-\frac 32}, \qquad \forall\,
\lambda \ge 10,
}
\end{align}
where $C_{K,\nu}>0$ depends on ($K$, $\nu$), and $\alpha_{j,\nu}$,
$\beta_{j,\nu}$ are computable coefficients.

Consider dimension $d\ge 2$ and denote by $d\sigma =d\sigma (\omega)$ the standard spherical
measure on the unit-sphere $\mathbb S^{d-1}=\{ \omega \in \mathbb R^d:\; |\omega|=1 \}$.
Denote $e_1=(1,0,\cdots,0)^{T}$.
Then for $\lambda>0$,
\begin{equation*}
\begin{aligned}
F_{d}(\lambda) &= \int_{\mathbb S^{d-1}} e^{i \lambda \omega \cdot e_1} d\sigma(\omega)  = c_d^{(1)} \int_0^{\pi} e^{i \lambda \cos \phi_1} \sin^{d-2} \phi_1 d\phi_1  \\
&= c_d^{(2)} \int_{-1}^1 e^{i \lambda t} (1-t^2)^{\frac {d-3}2} dt  = c_d^{(3)} \lambda^{-\frac {d-2}2} J_{\frac {d-2}2}(\lambda),
\end{aligned}
\end{equation*}
where $c_d^{(1)}>0$, $c_d^{(2)}>0$, $c_d^{(3)}>0$ are constants depending only on the dimension $d$.

By \eqref{2.15a}, we obtain (below $a_{j,d}$, $b_{j,d}$ are coefficients)
for any integer $K\ge 1$
\begin{align} \label{2.20a}
\boxed{
\left|F_{d}(\lambda)
- \lambda^{-\frac {d-1}2}
\cos \lambda \sum_{j=0}^K a_{j,d} \lambda^{-j}
-\lambda^{-\frac {d-1}2}
\sin \lambda \sum_{j=0}^K b_{j,d} \lambda^{-j}
\right| \le \tilde C_{K,d} \lambda^{-K-\frac {d+1}2},
\qquad \forall\, \lambda\ge 10,}
\end{align}
where $\tilde C_{K,d}>0$ depends only on ($K$, $d$).

\section{Proof of Theorem \ref{thm1} and Theorem \ref{thm2}}
To prove Theorem \ref{thm1}, we only need to prove the following theorem.
\begin{thm} \label{t4b}
Let $\Phi$ satisfy the conditions (a)--(c) (see \eqref{2.25aa}--\eqref{2.25a}).
Suppose  $u \in L^{\infty}(\mathbb R^d)$ and
satisfy
\begin{align}
\label{3.1a}
\left\langle u, ~\mathcal F^{-1}\Bigl(  (\Phi(|\xi|^2)-\Phi(1) ) \phi(\xi)   \Bigr) \right\rangle =0, \qquad \forall\, \phi
\in \mathcal S(\mathbb R^d).
\end{align}
Then the following hold:

\begin{enumerate}
\item $\mathrm{supp}(\widehat{u} ) \subset
K=\{ \xi:\; \text{$\xi=0$ or $|\xi|=1$} \}.$ More precisely we have
\begin{align} \label{3.2a}
\langle \widehat{u}, \phi \rangle =0, \qquad  \forall\, \phi \in C_c^{\infty}( \mathbb R^d \setminus  K).
\end{align}

\item If $\Phi(1) \ne 0$, then $ \mathrm{supp}(\widehat{u} )
\subset \{ \xi: \, |\xi| =1 \}$, and
\begin{align} \label{3.3a}
\langle \widehat{u}, \phi \rangle =0, \qquad
\forall\, \phi \in C_c^{\infty}(\mathbb R^d \setminus \{ \xi:\, |\xi|=1 \} ).
\end{align}
Furthermore we have in this case,
\begin{align}
\left\langle u, \mathcal F^{-1}\Bigl(  ( |\xi|^2-1) \psi(\xi)   \Bigr) \right\rangle =0, \qquad \forall\, \psi \in \mathcal S(\mathbb R^d).
\end{align}
\end{enumerate}
\end{thm}
\begin{proof}
We sketch the details.

(1)  Consider $\phi \in C_c^{\infty}( \mathbb R^d \setminus \{ \text{$\xi=0$ or $|\xi|=1$}\})$. Clearly we
 have the decomposition
 \begin{align*}
 \phi = \phi_1 + \phi_2,
 \end{align*}
 where $\phi_1 \in C_c^{\infty}(\{\xi:\, 0<|\xi|<1 \})$ and $\phi_2 \in C_c^{\infty}(\{\xi:\, |\xi|>1\})$.
 With no loss we may assume that
 $\phi_1 \in C_c^{\infty}(\{\xi:\, \delta_1<|\xi|<1-\delta_1 \})$ and
{$\phi_2 \in C_c^{\infty}(\{\xi:\, 1+\delta_1<|\xi|<\frac 1 {\delta_1} \})$} for some $\delta_1>0$ sufficiently small.
 This assumption is harmless since $\phi_1$ and $\phi_2$ are both compactly supported.

 By our assumption (c) on the function $\Phi$, we have
 \begin{align}
 {\sup_{\substack{\delta_1<|\xi|<1-\delta_1\\
 \text{or }1+\delta_1<|\xi|<\frac 1 {\delta_1} }}
 \frac 1 {|\Phi(|\xi|^2) - \Phi(1) |} \lesssim 1.}
 \end{align}
 This is because $\Phi$ is smooth and $\Phi(t) \ne \Phi(1)$ for any $t\in (0,1) \cup (1,\infty)$.

It is then  not difficult to check that
 $$
 \frac{\phi_1(\xi)}{\Phi(|\xi|^2)-\Phi(1)}\in C_c^\infty(\{\xi:\, 0<|\xi|<1 \}),
 \quad \frac{\phi_2(\xi)}{\Phi(|\xi|^2)-\Phi(1)}\in C_c^\infty(\{\xi:\, |\xi|>1 \}).
 $$
 Then
\begin{equation*}
\begin{aligned}
 &\langle \widehat u, \phi_1 \rangle
 =\left\langle \widehat u,
 (\Phi(|\xi|^2)-\Phi(1))\cdot\frac{\phi_1(\xi)}{\Phi(|\xi|^2)-\Phi(1)}\right\rangle=0; \notag \\
& \langle \widehat u, \phi_2 \rangle
 =\left\langle \widehat u,
 (\Phi(|\xi|^2)-\Phi(1))\cdot\frac{\phi_2(\xi)}{\Phi(|\xi|^2)-\Phi(1)}\right\rangle=0.
 \end{aligned}
\end{equation*}
Thus \eqref{3.2a} holds.

(2). The case $\Phi(1) \ne 0$.  Choose $\chi \in C_c^{\infty}(\mathbb R^d)$
such that $\chi(z) =1$ for $|z| \le \frac 12$ and $\chi(z)=0$ for $|z| \ge 1$.
Clearly
\begin{align*}
& \left\langle u, \mathcal F^{-1}\left( (\Phi(|\xi|^2) -\Phi(1) ) \chi\left(\frac {\xi}{\varepsilon} \right)  \psi(\xi) \right) \right\rangle =0;
\qquad (\text{by \eqref{3.1a}})\notag \\
&\lim_{\varepsilon\to 0}\left \langle u, \mathcal F^{-1} \left(  \Phi(|\xi|^2) \chi\left(\frac {\xi} {\varepsilon} \right)
\psi (\xi) \right) \right\rangle=0;   \qquad (\text{by Proposition \ref{pr2.1}}) \notag \\
&\lim_{\varepsilon\to 0}\left \langle u, \mathcal F^{-1} \left(  |\xi|^2\chi\left(\frac {\xi} {\varepsilon} \right)
\psi (\xi) \right) \right\rangle=0.   \qquad (\text{obvious}) \notag \\
\end{align*}
A suitable linear combination of the above yields (here we use $\Phi(1) \ne 0$)
\begin{equation*}
\lim\limits_{\varepsilon \to 0}
\left\langle u, \mathcal F^{-1}\Bigl(  ( |\xi|^2-1) \chi \left( \frac {\xi}{\varepsilon} \right) \psi (\xi ) \Bigr) \right\rangle =0.
\end{equation*}
On the other hand, it is not difficult to check that
\begin{align*}
\lim_{\varepsilon \to 0}
\left\langle u,
\mathcal F^{-1} \Bigl( (|\xi|^2-1)  (1- \chi(\varepsilon \xi) ) \psi (\xi) \Bigr) \right \rangle=0.
\end{align*}
Note that $\left(1-\chi\left(\frac {\xi}{\varepsilon}\right) \right) (1-\chi(\varepsilon \xi) )=
1-\chi(\varepsilon \xi)$. Thus
\begin{align}
\lim_{\varepsilon \to 0}
\left\langle u,
\mathcal F^{-1} \Bigl( (|\xi|^2-1)  \left(1-\chi\left(\frac {\xi}{\varepsilon} \right) \right)(1- \chi(\varepsilon \xi) ) \psi (\xi) \Bigr) \right \rangle=0.
\end{align}
We now only need to check for each small $\varepsilon>0$ the identity
\begin{align} \label{3.12a}
\left \langle u, \mathcal F^{-1} \Bigl( (|\xi|^2-1)
\underbrace{ \left(1-\chi\left(\frac{\xi}{\varepsilon}\right)\right) \chi(\varepsilon \xi) \psi (\xi) }_{\psi_{\varepsilon}}  \Bigr) \right \rangle=0.
\end{align}
Observe that $\psi_{\varepsilon} \in C_c^{\infty}$ and
\begin{align} \label{3.13a}
( |\xi|^2-1) \psi_{\varepsilon} (\xi)
 = (\Phi(|\xi|^2)- \Phi(1) ) \cdot \underbrace{\frac { |\xi|^2-1} {\Phi(|\xi|^2) -\Phi(1) }
 \psi_{\varepsilon} (\xi)}_{\text{$\in C_c^{\infty}(\mathbb R^d)$}}.
 \end{align}
 Here we use the crucial assumption (c) on $\Phi$, namely: 1) near $|\xi|=1$, $\Phi^{\prime}(1) \ne 0$;
 2) away from $|\xi|=1$ (and in a compact neighborhood of $|\xi|=1$), $|\Phi(|\xi|^2) -\Phi(1)|\gtrsim 1$.
 These two facts yield that $\frac {|\xi|^2-1} {\Phi(|\xi|^2)-\Phi(1)}$  can be defined as a smooth function in
 the whole neighborhood of $|\xi|=1$.

 Thus \eqref{3.12a} holds and we have
\begin{align}
\left\langle u, \mathcal F^{-1}\Bigl(  ( |\xi|^2-1) \psi(\xi)   \Bigr) \right\rangle =0, \qquad \forall\, \psi \in \mathcal S(\mathbb R^d).
\end{align}
The statement \eqref{3.3a} can be proved along similar lines. We omit the details.
\end{proof}

\begin{proof}[Proof of Theorem \ref{thm2}]
The main modification is in \eqref{3.13a}:
\begin{align}
( |\xi|^2-1)^{j_0} \psi_{\varepsilon} (\xi)
 = (\Phi(|\xi|^2)- \Phi(1) ) \cdot \underbrace{\frac {( |\xi|^2-1)^{j_0}} {\Phi(|\xi|^2) -\Phi(1) }
 \psi_{\varepsilon} (\xi)}_{\text{$\in C_c^{\infty}(\mathbb R^d)$}}.
 \end{align}
Clearly the result follows.
\end{proof}

\section{Connection with the Bernstein Helmholtz case in \cite{gmw2022}}
We now show that in the more general Bernstein Helmholtz case introduced in \cite{gmw2022},
the conditions on $\Phi(\lambda)$ in \cite{gmw2022} are stronger than
our conditions on the function $\Phi(\lambda)$.

Recall that in \cite{gmw2022}, one assumes that $\Phi(\lambda)$ is a complete Bernstein function,
and in the corresponding harmonic extension problem, the weight function $a(t) \in A_2$ (in particular
$a$ is locally integrable) and satisfies
$a(t) \sim t^{\alpha}$ for $t\gg 1$ where $|\alpha|<1$.

We shall show that in \cite{gmw2022},  as long as
$\Phi$ is complete Bernstein, $a$ is weakly integrable, $a(t) \sim t^{\alpha}$ for
$t\gg 1$ where $|\alpha|<1$, then such $\Phi$ will satisfy our conditions (a)--(c)
(see \eqref{2.25aa}--\eqref{2.25a}) with the property $\Phi(1) \ne 0$.

1) Since $\Phi(\cdot):\, [0, \infty) \to [0, \infty)$ is complete Bernstein, we have
\begin{align} \label{2.42a}
	\Phi(\lambda) = c_1 + c_2 \lambda + \int_{ (0,\infty)} \frac {\lambda}{\lambda+s} \frac {m(ds)}s,
\end{align}
for some constants $c_1\ge 0$, $c_2\ge 0$, and the nonnegative measure $m$ satisfies
\begin{align} \label{2.42b}
	\int_{(0,\infty)} \frac 1 {1+s} \frac {m(ds)}{s} <\infty.
\end{align}
Clearly $\Phi \in C^{\infty} ( (0,\infty) ) \cap C([0, \infty) )$.  It is easy to check that
\eqref{2.25aa} holds.

2) We check \eqref{2.25a}.  Although in general the exact profile of the Krein correspondence
$\Phi \leftrightarrow a$ is hard to determine,  we can work out the asymptotic information
via the quadratic form  inequality (cf. Theorem II of \cite{KM17}, note that the convention
of Fourier transform therein differs from ours by a constant) : namely
\begin{align}
	\int_0^{\infty} \int_{\mathbb R} a(t) ( |\partial_t u|^2 + |\partial_x u |^2) dx dt
	\ge  \mathrm{const} \cdot \int_{\mathbb R} \Phi(|\xi|^2) |\widehat f(\xi) |^2 d\xi,
\end{align}
where $u(0, x) =f (x)$.  Note that the equality (with sharp constants) is achieved when $u$ is a suitable
harmonic extension of $f$.

Setting $\widehat u(t, \xi) = e^{-t |\xi|} \widehat f(\xi)$, we obtain
\begin{align}
	\int_{\mathbb R} \Bigl(\int_0^{\infty} a(t) e^{-2t |\xi| } dt \Bigr) |\xi|^2 |\widehat f(\xi)|^2 d\xi
	\gtrsim \int_{\mathbb R} \Phi(|\xi|^2) |\widehat f(\xi)|^2 d\xi.
\end{align}
Since by assumption $a(t) \sim t^{\alpha}$ for $t\gg 1$ and $|\alpha|<1$, we have
(below $R_0\gg 1$ is the constant for which $a(t) \sim t^{\alpha}$ when $t\ge R_0$)
\begin{align}
	\int_0^{R_0} a(t) dt \int_{\mathbb R} |\xi|^2 |\widehat f(\xi)|^2 d \xi
	+ \int_{\mathbb R} |\xi|^{1-\alpha} |\widehat f(\xi)|^2 d\xi
	\gtrsim \int_{\mathbb R} \Phi(|\xi|^2) |\widehat f(\xi)|^2 d\xi.
\end{align}

Since by assumption $a$ is locally integrable, we have $\int_0^{R_0} a(t) dt \lesssim 1$.
Since $|\alpha|<1$, we have for all $\widehat f$ with support in $\{\xi:  |\xi|<1 \}$,
\begin{align}
	\int_{|\xi|<1} |\xi|^{1-\alpha} |\widehat f(\xi)|^2 d\xi
	\gtrsim \int_{|\xi|<1} \Phi(|\xi|^2) |\widehat f(\xi)|^2 d\xi.
\end{align}
By choosing
suitable $\hat f (\xi) \sim |\xi|^{-1+\delta}$ ($2\delta>\alpha$)  when $|\xi| \ll 1$, we obtain
\begin{align*}
	\int_{|\xi|\ll 1} \Phi(|\xi|^2) |\xi|^{-2+2\delta} d\xi  \lesssim 1.
\end{align*}
Here we note that $\alpha<1$, and we can choose $2\delta = 1-\eta$ for some $\eta >0$ sufficiently small.
This easily implies
\begin{align*}
	\int_0^{\varepsilon_0} |\Phi(z)| |z|^{-1}  dz \lesssim 1.
\end{align*}
Note that here we actually proved $c_1=0$ in \eqref{2.42a}.
By \eqref{2.42a}, we have for $\lambda>0$,
\begin{align}
	& \Phi(\lambda) = c_1 + c_2 \lambda + \int_{(0,\infty)} \left(1- \frac s {\lambda+s} \right) \frac {m(ds)} s; \\
	& \Phi^{\prime}(\lambda) =c_2 + \int_{(0,\infty)} \frac s {(\lambda+s)^2} \frac {m(ds)} s ; \label{2.51a}\\
	& \lambda |\Phi^{\prime}(\lambda)| \le c_2 \lambda+
	\int_{(0,\infty)} \frac {\lambda} {\lambda+s} \frac {s} {\lambda+s}
	\frac {m(ds)}s \le \Phi(\lambda).
\end{align}
Similar estimates hold for higher derivatives. Thus \eqref{2.25a} holds.

3).  We check the uni-valence of $\Phi$ at $\Phi(1)$.  First we show $\Phi(1) \ne 0$. By \eqref{2.42a}, we have
\begin{align}
	\Phi(1) = c_1 +c_2+ \int_{(0,\infty)} \frac 1 {1+s} \frac {m(ds)}s.
\end{align}
If $\Phi(1) =0$, then $c_1=c_2=0$, and $\int_{(0,\infty)} \frac 1 {1+s} \frac {m(ds)}s=0$. It follows
that $\Phi \equiv 0$ which contradicts to the assumption that $a(t) \sim t^{\alpha}$ for $t\gg 1$.

Next we show $\Phi^{\prime}(1) \ne 0$.  Suppose $\Phi^{\prime}(1)=0$. By \eqref{2.51a}, we obtain
\begin{align}
	c_2=0, \qquad \int_{(0,\infty)} \frac 1 {(1+s)^2} m(ds) =0.
\end{align}
By using Lebesgue monotone convergence, we have
\begin{align*}
	\int_{(0,\infty)} \frac {1} {(1+s)} \frac {m(ds)} s = \lim_{\varepsilon \to 0+}
	\int_{(\varepsilon, \infty)} \frac 1 {1+s} \frac {m(ds)} s =0.
\end{align*}
This implies that {$\Phi \equiv c_1=\Phi(0)=0$}.  Thus we rule out this possibility and conclude $\Phi^{\prime}(1) \ne 0$.

Finally we observe that $\Phi^{\prime}(1) \ne 0$ and $\Phi \in C^{\infty} ( (0, \infty) )$.
Clearly $\Phi $ is strictly monotone near $\lambda=1$.
By monotonicity we have for $\delta_1>0$ sufficiently small,
\begin{align}
	\max_{0\le \lambda \le 1-\delta_1} \Phi(\lambda) \le \Phi( 1- \delta_1) <\Phi(1) ,
	\qquad \Phi(1)<\Phi(1+\delta_1) \le \Phi (\lambda), \;\forall\, \lambda\ge 1+\delta_1.
\end{align}
Thus $\Phi$ satisfies our condition (c).

%
%

%

\end{document}